\documentclass[final]{siamltex}

\usepackage{amsmath,amssymb,amsxtra,comment,graphicx,psfrag}
\usepackage{bm,mathrsfs}
\usepackage{mathtools}
\usepackage{color,xcolor}
\usepackage{enumerate,enumitem}
\usepackage{pdfsync}
\usepackage{hhline}
\usepackage{cite}
\usepackage{bm}
\usepackage[linkcolor=blue,
            citecolor=forestgreen]{hyperref}

\newtheorem{remark}{Remark}[section]
\newtheorem{example}{Example}[section]

\newcommand{\R}{\mathbb{R}}

\def\d{{\rm d}}
\def\II{(\Omega)}

\title{Arbitrarily high-order exponential cut-off methods for preserving maximum principle\\ of parabolic equations}

\author{
Buyang Li\thanks{Department of Applied Mathematics, The Hong Kong Polytechnic University, Hung Hom, Hong Kong. E-mail address: buyang.li@polyu.edu.hk}
\and
Jiang Yang\thanks{Department of Mathematics \& SUSTech International Center for Mathematics, Southern University of Science and Technology, Shenzhen 518055, China. E-mail address: yangj7@sustech.edu.cn}
\and
Zhi Zhou\thanks{Department of Applied Mathematics, The Hong Kong Polytechnic University, Hung Hom, Hong Kong. E-mail address: zhizhou@polyu.edu.hk}
}

\begin{document}

\maketitle

\begin{abstract}
A new class of high-order maximum principle preserving numerical methods is proposed for solving parabolic equations, with application to the semilinear Allen--Cahn equation. The proposed method consists of a $k$th-order multistep exponential integrator in time, and a lumped mass finite element method in space with piecewise $r$th-order polynomials and Gauss--Lobatto quadrature. At every time level, the extra values violating the maximum principle are eliminated at the finite element nodal points by a cut-off operation.
The remaining values at the nodal points satisfy the maximum principle and are proved to be convergent with an error bound of $O(\tau^k+h^r)$. The accuracy can be made arbitrarily high-order by choosing large $k$ and $r$.
Extensive numerical results are provided to illustrate the accuracy of the proposed method and the effectiveness in capturing the pattern of phase-field problems.
\end{abstract}



\section{Introduction}

The evolution of physical quantities, such as the density, concentration, pressure, can often be described by time-dependent parabolic partial differential equations (PDEs). In this case, the solutions of these PDEs can only take values in a given range to be consistent with physical phenomena. Mathematically, this property is often guaranteed by the maximum principle of the parabolic PDEs. In the numerical simulation, it is also desired to guarantee that the numerical solutions only take values in the given range to be consistent with physical phenomena and to avoid producing spurious solutions. Correspondingly, great efforts have been made in developing and analyzing numerical methods that preserve maximum principle in the discrete setting.

It is well known that the standard backward Euler time-stepping scheme and
central finite difference method in space can preserve the maximum principle of linear parabolic equations \cite[Chapter 9]{Larsson-Thomee-2003}.
It is also known that the backward Euler time-stepping scheme with lumped mass linear finite element method (FEM) in space, using simplicial triangulation with  acute angles, can preserve the maximum principle \cite{Fujii-1973}.
This generalizes in two-dimensional space to triangulations of Delauney type, which is proved to be essentially sharp  \cite{Thomee:2008}.
Without using mass lumping, the standard Galerkin FEMs normally do not preserve the  maximum principle of parabolic equations \cite{Thomee-2015, Thomee:2008, Schatz:2010}.
See also \cite{Chat:2015} for an analysis
of finite volume element method. These methods are all first-order in time and second-order in space.



In recent years, many efforts were made in constructing maximum principle preserving (MPP) methods for the initial-boundary value problems of the Allen--Cahn phase field equation
\begin{align}
\left\{
\begin{aligned}
&\frac{\partial u}{\partial t}= \Delta u+ f(u) &&\mbox{in}\,\,\,\Omega\times (0,T],\\[3pt]
&\partial_nu=0 &&\mbox{on}\,\,\,\partial\Omega\times (0,T],\\[2pt]
&u|_{t=0}=u_0 &&\mbox{in}\,\,\,\Omega,
\end{aligned}
\right.
\end{align}
where $f(u)=-F'(u)$ with a double-well potential $F$ that has two wells at $\pm\alpha$, for {   some known parameter $\alpha>0$}; see Examples \ref{AC:1} and \ref{AC:2}
for two popular choices of potentials. It is well-known that the Allen--Cahn equation has the maximum principle: if the initial value satisfies $|u_0|\le \alpha$ then the solution can take values only in the interval $[-\alpha,\alpha]$ at later time.
In \cite{tangyang1,ShenTangYang:2016}, it was proved that the stabilized backward Euler time-stepping scheme,
with central difference method in space,  preserves the maximum principle under certain time stepsize restriction.
%
More recently, a stabilized exponential time differencing scheme was proposed in \cite{DuJuLiQiao:2019} for solving the nonlocal Allen--Cahn equation. For local Allen--Cahn equation, their spatial discretization method automatically reduces to the second-order central difference method. The scheme preserves the maximum principle under certain time stepsize restriction related to the nonlinearity of the Allen--Cahn equation, with second-order convergence in both time and space.
{   The argument in \cite{DuJuLiQiao:2019} was generalized to a class of semilinear parabolic equations in \cite{DuJuLiQiao:2020}.}
As far as we know, this is the highest-order linearly implicit MMP method for the Allen--Cahn equation without stepsize conditions.
{    A second-order MMP
backward differentiation formula with nonuniform meshes for the Allen-Cahn equation was analyzed in \cite{LiaoTangZhou:2020} under some practical stepsize constraints.}
The construction of higher-order MPP methods for the Allen--Cahn equation is still challenging.

A cut-off finite difference method was proposed in \cite{Lu-Huang-Vleck-2013} for linear parabolic equations. The method cuts off the extra values which violate the maximum principle, resulting in a solution which satisfies the maximum principle. It is mentioned that the $\theta$-method in time with second-order central difference method in space satisfies the conditions in \cite{Lu-Huang-Vleck-2013} and therefore are MPP methods.

A third-order spatially semidiscrete discontinuous Galerkin MMP method was proposed in \cite{Liu-Yu-2014}. The method can be combined with high-order strong stability preserving (SSP) time-stepping methods to preserve maximum principle under the parabolic CFL condition $\tau=o(h^2)$. Such idea originates from solving hyperbolic conservation laws, for which a limiter can be used to control the solution values at discrete points given by high-order spatial discretization and SSP temporal discretization, under the hyperbolic CFL condition $\tau=o(h)$; see e.g., \cite{GottliebShuTadmor:2001, Jiang-Xu-2013,Liu-Osher-1996,Qiu-Shu-2005,Xu-2014,Zhang-Shu-2010} for a rather incomplete list of references, and a comprehensive overview of SSP method in \cite{Gottlieb-2011}.

Recently, an SSP integrating factor Runge--Kutta method of up to order $4$ was proposed and analyzed in \cite{IsherwoodGrantGottlieb:2018} for semilinear hyperbolic and parabolic equations. For semilinear hyperbolic and parabolic equations with strong stability (possibly in the maximum norm), the method can preserve this property and avoids the standard parabolic CFL condition $\tau=o(h^2)$, only requiring the stepsize $\tau$ to be smaller than some constant depending on the nonlinear source term. A nonlinear constraint limiter was introduced in \cite{Vegt-Xia-Xu-2019} for implicit time-stepping schemes without requiring CFL conditions, which can preserve maximum principle at the discrete level with arbitrarily high-order methods by solving a nonlinearly implicit system.

In this article, we propose a fully discrete high-order MPP method for semilinear parabolic equations, which we refer to as the exponential cut-off method. The method consists of a $k$th-order multistep exponential integrator in time and a $r$th-order lumped mass FEM in space, together with a cut-off operation to eliminate the extra values violating the maximum principle before proceeding to the next time level. At every time level, one only needs to solve a linear system of equations. We prove that the proposed method has at least $k$th-order convergence in time and $r$th-order convergence in space, without restriction on the temporal stepsize and spatial mesh size. Thus the proposed method can be made arbitrarily high-order by choosing large $k$ and $r$. We also provide numerical results to show the effectiveness of the proposed method in approximating the logarithmic Allen--Cahn equation, for which the numerical solution may blow up if the maximum principle is not preserved.

{   
The classical exponential integrators for an abstract semilinear parabolic equation
\begin{align}\label{abstract-parabolic}
\partial_tu - A u = f(t,u)
\end{align}
were developed and analyzed by many authors, including the one-step exponential Runge--Kutta methods and the exponential multistep methods.
The construction of exponential Runge--Kutta methods were developed in \cite{Lawson-1967,Ehle-Lawson-1975,Strehmel-Weiner-1987,Cox-Matthews-2002} and \cite{Hochbruck-Ostermann-2005b}. The latter contains a general principle for constructing exponential Runge--Kutta methods.
This exponential Runge--Kutta methods are multi-stage methods that reduce to the classical Runge--Kutta methods when $A=0$.
Instead of the exponential Runge--Kutta methods, we adopt the exponential multistep methods as the underlying time-stepping method (for the cut-off postprocess), which approximates $f(t,u)$ by a backward extrapolation polynomial and results in a linearly implicit method without interval stages. This class of methods were proposed in \cite{Certaine-1960} and studied more systematically in \cite{Norsett-1969,Cox-Matthews-2002,Calvo-Palencia-2006}. Since this method does not have internal stages, it is relatively easier for both implementation and analysis.
}

The rest of this article is organized as follows. In sections \ref{section:linear} and \ref{sec:AC}, we present the cut-off exponential lumped mass method for linear parabolic equations and the Allen--Cahn equation, respectively, and present error estimates for the methods. In section \ref{section:numerical}, we present numerical results to show the accuracy of the proposed method, and illustrate the effectiveness of the method in solving the Allen--Cahn equation with different nonlinear potentials in comparison with other MPP methods.

\section{Linear parabolic equations}\label{section:linear}

In this section, we consider the following initial-boundary value problem of a linear parabolic equation:
\begin{equation}\label{eqn:linear}
    \left\{\begin{aligned} \partial_t  u -  \Delta u&= f &&
    \mbox{in}\,\,\,\Omega\times(0,T],  \\
\partial_n u &= 0 && \mbox{on}\,\,\,\partial\Omega\times(0,T], \\
    u|_{t=0} &= u_0  &&\mbox{in}\,\,\, \Omega, \\
   \end{aligned} \right.
\end{equation}
where $\Omega\subset\mathbb{R}^d$ is a $d$-dimensional rectangular domain with boundary $\partial\Omega$. Equation \eqref{eqn:linear} satisfies the weak maximum principle, i.e., if $f$ is nonnegative then {   the solution of \eqref{eqn:linear} satisfies (cf. \cite[Chapter 8]{Larsson-Thomee-2003})}
\begin{align}\label{MP-linear}
\min_{x\in\overline\Omega} u(x,t) \ge  u_{\min}
\quad \mbox{for}\,\,\, t\in(0,T] ,
\end{align}
with
$$
u_{\min}=\min_{x\in\overline\Omega}u_0(x) .
$$
In this section, we propose a high-order fully discrete MMP method to preserve property \eqref{MP-linear} in the discrete level.

\subsection{Exponential cut-off methods: 1D case}$\,$
In this subsection, we consider the one-dimensional case $\Omega=[a,b]$. The extension to multi-dimensional cases is presented in subsection \ref{section:exp-LM-mD}.

We denote by $a=x_0<x_{1}<\dots<x_{Mr}=b$ a partition of the domain with a uniform mesh size $h=x_{ir} - x_{(i-1)r} = (b-a)/M$,
and denote by $S_h^r$ the finite element space of degree $r\ge 1$, i.e.,
$$
S_h^r=\{v\in H^1(\Omega): v|_{I_i}\in P_r,\,\,\, i=1,\dots,M\} ,
$$
where $I_i=[x_{(i-1)r},x_{ir}]$ and $P_r$ denotes the space of polynomials of degree $\le r$.

Let $x_{(i-1)r+j}$ and { $\omega _{j}$}, $j=0,\dots,r$, be the quadrature
points and weights of the $(r+1)$-point Gauss--Lobatto quadrature on the subinterval $I_i$, and denote
\begin{align}\label{eqn:weight}
w_{(i-1)r+j} =
\left\{
\begin{aligned}
&\omega_j &&\mbox{for}\,\,\, 1\le j\le r-1,\\
&2\omega_j &&\mbox{for}\,\,\, j=0, r.
\end{aligned}
\right.
\end{align}
Then  we consider the piecewise Gauss--Lobatto quadrature approximation of the inner product, i.e.,
$$
(f,g)_h
:=  \sum_{j=0}^{Mr} w_j f(x_j) g(x_j)  .
$$
This discrete inner product induces a norm
$$
\|f_h\|_h=\sqrt{(f_h,f_h)_h} \quad\forall\, f_h\in S_h^r.
$$

For $(Mr+1)$-dimensional vectors ${\bf u}$ and ${\bf v}$, we define the discrete $L^2$ inner product and discrete $L^2$ norm by
$$
({\bf u},{\bf v})_h=  {\bf M}{\bf u}\cdot {\bf v}
\quad\mbox{and}\quad
\|{\bf v}\|_{h}
= \sqrt{({\bf v},{\bf v})_h} \, ,
$$
where ${\bf M}$ is the $(Mr+1)\times (Mr+1)$ mass matrix, a diagonal matrix consisting of the weights of the piecewise Gauss--Lobatto quadrature corresponding to the quadrature points. Then we have the following lemma for norm equivalence. The proof follows directly from the positivity of
{   Gauss--Lobatto} quadrature weights \cite[p. 426]{Quarteroni:2000},
and the equivalence of finite dimensional norms.

\begin{lemma}\label{lem:norm-equivalence}
If $v_h\in S_h^r$ and ${\bf v}$ is the $(Mr+1)$-dimensional vector consisting of the nodal values of $v_h$,
then the following three norms are equivalent in sense that
\begin{align*}
\|v_h\|_h = \|{\bf v}\|_h \sim \|v_h\|_{L^2(\Omega)} .
\end{align*}
\end{lemma}


The solution of \eqref{eqn:linear} satisfies the weak form
\begin{align}\label{weak}
(\partial_t u , v) + (\partial_x u , \partial_x v) = (f , v) \quad \forall\, v\in H^1\II,
\end{align}
which implies
\begin{align}\label{weak-numer-int}
(\partial_t u, v_h)_h + (\partial_x \Pi_h u, \partial_x v_h)  = (f, v_h)_h + \mathcal{E}(v_h) \quad\forall\, v_h\in S_h^r,
\end{align}
where $\Pi_h:C(\overline\Omega)\rightarrow S_h^r$ is the Lagrange interpolation operator, and $\mathcal{E}(v_h)$ includes both quadrature and interpolation errors. Since the $(r+1)$-point Gauss--Lobatto quadrature on each subinterval $I_i$ is accurate for polynomials of degree $2r-1$ on $I_i$ \cite[p. 425]{Quarteroni:2000}, employing the Bramble--Hilbert lemma, we conclude that
\begin{align}\label{GL-quadrature-error}
|\mathcal{E}(v_h)|
&=|(\partial_t u, v_h)_h-(\partial_t u, v_h)
+(f, v_h)-(f, v_h)_h
+ (\partial_x (\Pi_h u-u), \partial_x {   v_h}) |\notag \\
&\le
c\sum_{i=1}^{M}h^{2r}\|\partial_t u \, v_h\|_{W^{2r,1}(I_i)}
+c\sum_{i=1}^{M}h^{2r}\|fv_h\|_{W^{2r,1}(I_i)}
+ |(\partial_x(u-\Pi_hu),\partial_xv_h)| \notag \\
&\le
ch^{2r} \sum_{i=1}^{M} \sum_{j=0}^r\|\partial_t u\|_{H^{2r-j}(I_i)}\|v_h\|_{H^j(I_i)}
+ c\sum_{i=1}^{M}h^{2r} \sum_{j=0}^r \|f\|_{H^{2r-j}(I_i)}\|v_h\|_{H^j(I_i)} \notag \\
&\quad\,
+ch^{r}\|u\|_{H^{r+1}(\Omega)}\|\partial_x v_h\|_{L^2(\Omega)}
\notag \\
&\le
ch^{2r} \sum_{i=1}^{M}(\|\partial_t u\|_{H^{2r}(I_i)}+\|f\|_{H^{2r}(I_i)})\|v_h\|_{H^r(I_i)}
+ch^{r}\|u\|_{H^{r+1}(\Omega)}\|\partial_x v_h\|_{L^2(\Omega)} \notag \\
&\le
C(\|\partial_t u\|_{H^{2r}(\Omega)}+\|u\|_{H^{r+1}(\Omega)}+\|f\|_{H^{2r}(\Omega)})
(\|v_h\|_h +\|\partial_x v_h\|_{L^2(\Omega)}) h^r ,
\end{align}
where the last inequality follows from the inverse inequality of the finite element space and the norm equivalence, cf. Lemma \ref{lem:norm-equivalence}.

The spatially semi-discrete Gauss--Lobatto lumped mass method is to find $u_h\in S_h^r$ satisfying the following equation:
\begin{align}\label{semi-discrete-GL}
(\partial_t u_h , v_h)_h
+ (\partial_x u_h, \partial_x v_h) = (f, v_h)_h
\quad v_h\in S_h^r .
\end{align}
This can be furthermore written into a matrix-vector form:
\begin{align}\label{semi-discrete-GL-matrix}
{\bf M}\dot {\bf u} - {\bf A} {\bf u}
={\bf M}{\bf f}
 \quad
\text{or equivalently}
\quad
\dot {\bf u} - {\bf M}^{-1}{\bf A} {\bf u}
={\bf f}
\end{align}
where ${\bf u}$ is a $(Mr+1)$-dimensional vector consisting of the nodal values of $u_h$, and $\dot{\bf u}$ denotes the time derivative of the vector ${\bf u}$; ${\bf M}$ and ${\bf A}$ are the mass and stiffness matrices, respectively.

For the time discretization of \eqref{semi-discrete-GL} or \eqref{semi-discrete-GL-matrix}, we consider the following {   exponential integrator
\cite{Hochbruck-Ostermann-2010, MinchevWright:2005}}
\begin{align}\label{exp-GL}
\hat{\bf u}^n = e^{\tau {\bf M}^{-1}{\bf A}}{\bf u}^{n-1} + \int_{t_{n-1}}^{t_n}e^{(t_n-s) {\bf M}^{-1}{\bf A}} I_\tau^{(k-1)}{\bf f}(s)  \,ds .
\end{align}
{   Here, we approximate the function $f(s)$, on $[t_{n-1},t_n]$, by the extrapolation polynomial
\begin{align}\label{eqn:extrapolation}
  I_\tau^{(k-1)}{\bf f}(s) =    \sum_{j=1}^k L_j(s) {\bf f} (t_{n-j}),
  \end{align}
where $L_j(s)$ is the Lagrange basis polynomials of degree $k-1$ in time, satisfying
$$
L_j(t_{n-i}) = \delta_{ij}, \quad i,j=1,\dots,k.
$$
If $f$ is smooth in $[0,T]$ then it can be smoothly extended to $t\le 0$ to define $I_\tau^{(k-1)}{\bf f}(s)$ for $s\in[0,t_{k-1}]$.}
Since $I_\tau^{(k-1)}{\bf f}$ is a polynomial in time, the integral in \eqref{exp-GL} can be evaluated exactly (up to errors in computing the exponential of matrices).

Then we truncate the nodal vector $\hat{\bf u}^n$ by setting
\begin{equation}\label{truncation-GL}
\begin{split}
{\bf u}^n & =  \max \big(\hat{\bf u}^n, u_{\min} {\bf 1} \big) ,
\end{split}
\end{equation}
where ${\bf 1}$ denotes the $(Mr+1)$-dimensional vector with element $1$ in each component.

For a given $u^{n-1}_h\in S_h^r$, the nodal vector ${\bf u}^{n-1}$ is uniquely determined. Then the nodal vector $\hat{\bf u}^n$ given by \eqref{exp-GL} consists of the nodal values of the solution $\hat u_h(t_n)$ obtained from the initial-value problem
\begin{align}\label{fully-discrete-GL}
\left\{\begin{aligned}
&(\partial_t \hat u_h , v_h)_h
+ (\partial_x \hat u_h, \partial_x v_h) = (I_\tau^{(k-1)}f, v_h)_h
\quad v_h\in S_h^r ,\,\,\, t\in(t_{n-1},t_n] \\
&\hat u_h(t_{n-1})=u^{n-1}_h .
\end{aligned}\right.
\end{align}
The nodal vector obtained from \eqref{truncation-GL} can be used to construct a piecewise polynomial $u^{n}_h\in S_h^r$.

The accuracy of the fully discrete scheme \eqref{exp-GL}-\eqref{truncation-GL} is presented in the following theorem.

\begin{theorem}\label{thm:exp-GL}
Let $u^{n}_h\in S_h^r$ be the piecewise polynomial corresponding to the nodal vector ${\bf u}^n$ obtained from the fully discrete scheme \eqref{exp-GL}-\eqref{truncation-GL}. Then
$$
 {u_h^n}(x_j)\ge u_{\min} , \quad j=0,1,\dots,Mr ,
$$
and
\begin{align*}
\max_{1\le n\le N} \| u(t_n) - u^n_h \|_{L^2(\Omega)} \le c (\tau^k+h^r) + c\| u(t_{0}) - u^{0}_h \|_{L^2(\Omega)} ,
\end{align*}
provided that {   $u\in C^1([0,T]; H^{2r}(\Omega))$ and $f \in C^k([0,T];L^2(\Omega))\cap C([0,T];H^{2r}(\Omega))$. }
\end{theorem}

\begin{proof}
Clearly, \eqref{truncation-GL} guarantees that $ {u_h^n}(x_j)\ge u_{\min}$ for $j=0,1,\dots,Mr$.
To prove the error estimate,  {we denote by $e_h^n = \Pi_h u(t_n) - u_h^n$ the
difference between the numerical solution and the Lagrange interpolation of the exact solution.
At each node $x_j$ there holds
\begin{align*}
|e_h^n(x_j)|
&= | \max \big( u(x_j, t_n) , u_{\min}) -  \max \big(\hat{u}_h(x_j,t_n), u_{\min} \big) |\\
&\le |u(x_j, t_n)  - \hat{u}_h(x_j,t_n)  | ,
\end{align*}
where the piecewise-defined function $\hat{u}_h$ is given by  \eqref{fully-discrete-GL}.
Then for $\hat e_h=\Pi_hu-\hat u_h$ there holds
\begin{align}\label{truncation-bfe}
\| e_h^n  \|_h
\le  \|  \hat e_h(t_n) \|_h \quad\mbox{and}\quad
\| e_h^{n-1} \|_h = \|\hat e_h(t_{n-1})\|_h
\end{align}}

Besides, the difference between \eqref{weak-numer-int} and \eqref{fully-discrete-GL}  yields the error equation
{   for $t\in (t_{n-1},t_n]$}
\begin{align}\label{error-eq-GL}
\left\{\begin{aligned}
&(\partial_t \hat e_h , v_h)_h
+ (\partial_x \hat e_h, \partial_x v_h) = (f-I_\tau^{(k-1)}f, v_h)_h
+\mathcal{E}(v_h)  \quad\forall\, v_h\in S_h^r , \\
&\hat e_h(t_{n-1})= \Pi_hu(t_{n-1})-u^{n-1}_h .
\end{aligned}\right.
\end{align}

Substituting $v_h=\hat e_h$ into the equation above, we obtain
\begin{align*}
& \frac{\d}{\d t}\bigg( \frac12 \|\hat e_h\|_h^2  \bigg)
+ \|\partial_x \hat e_h\|_{L^2(\Omega)}^2 \\
&= (f-I_\tau^{(k-1)}f, \hat e_h)_h
+\mathcal{E}(\hat e_h)  \\
&\le
\|f-I_\tau^{(k-1)}f\|_h \|\hat e_h\|_h
+ch^r(\|\hat e_h\|_h+\|\partial_x\hat e_h\|_{L^2(\Omega)})  \\
&\le
c(\tau^k + h^r) (\|\hat e_h\|_h+\|\partial_x\hat e_h\|_{L^2(\Omega)})  \\
&\le
c(\tau^k + h^r)^2 + \frac12\|\hat e_h\|_h^2 + \frac12\|\partial_x\hat e_h\|_{L^2(\Omega)}^2 ,
\end{align*}
which furthermore reduces to
\begin{align*}
\frac{\d}{\d t}\bigg( \frac12 \|\hat e_h\|_h^2  \bigg)
&\le
c(\tau^k + h^r)^2 + \frac12\|\hat e_h\|_h^2 .
\end{align*}
By applying Gronwall's inequality, we obtain
\begin{align*}
\|\hat e_h(t_n)\|_h^2
&\le e^{\tau} c\tau (\tau^k + h^r)^2 + e^{\tau} \|\hat e_h(t_{n-1})\|_h^2 \\
&\le c\tau (\tau^k + h^r)^2 + (1+ c\tau) \|\hat e_h(t_{n-1})\|_h^2.
\end{align*}
 {Then, 
using \eqref{truncation-bfe}, we have
\begin{align*}
\|e_h^n\|_h^2
&\le c\tau (\tau^k + h^r)^2 + (1+c\tau) \|e_h^n\|_h^2.
\end{align*}
Iterating the inequality above for $n=1,\dots,N$, we obtain
\begin{align*}
\max_{1\le n\le N}\|e_h^n\|_h^2
&\le e^{cT}\|e_h^{0}\|_h^2
+e^{cT}  c(\tau^k + h^r)^2 .
\end{align*}}
This and the norm equivalence in Lemma \eqref{lem:norm-equivalence} imply
\begin{align*}
&\max_{1\le n\le N} \| \Pi_h u(t_n) - u^n_h \|_{L^2(\Omega)} \\
&\le c (\tau^k+h^r) + c\| \Pi_h u(t_{0}) - u^{0}_h \|_{L^2(\Omega)} \\
&\le c (\tau^k+h^r) + c\| \Pi_h u(t_{0}) - u(t_{0}) \|_{L^2(\Omega)}
+ c\| u(t_{0}) - u^{0}_h \|_{L^2(\Omega)} \\
&\le c (\tau^k+h^r)
+ c\| u(t_{0}) - u^{0}_h \|_{L^2(\Omega)} .
\end{align*}
This proves the desired result in Theorem \ref{thm:exp-GL}.
\hfill\end{proof}

\subsection{Extension to multi-dimensional domains} \label{section:exp-LM-mD}
In the section, we describe the implementation of the proposed numerical method in a
multi-dimensional rectangular domain $\Omega = (a, b)^d \subset\R^d$, with $d\ge 2$.

For all $i=1,\dots, d$, we denote by $a=x_{0}<x_{1}<\dots<x_{M r}=b$
a partition of the interval $[a, b]$ with a uniform mesh size $\displaystyle h=x_{ir}-x_{(i-1)r} = (b-a)/M$
for all $i=1,\dots,M$. Using the same setting in one dimension, we let $x_{(i-1)r+j}$ and $w_{j}$,
$j=0,\dots,r$, be the quadrature points and weights of the $(r+1)$-point
Gauss--Lobatto quadrature on the subinterval $[x_{(i-1)r}, x_{ir}]$.
Moreover, we define  the global quadrature weights $w_i$, with $i=0,\ldots, Mr$, by \eqref{eqn:weight}.

The domain $\Omega$ is now separated into $M^d$ subrectangles by all grid points
$(x_{j_1 r},  \ldots, x_{j_dr})$,
with $0 \le j_i \le M$ and $i=1,\dots, d$. We denote this partition by $\mathcal{K}$,
and  note that $\displaystyle h$ is the mesh size of the partition $\mathcal{K}$.
Then we apply the tensor-product Lagrange finite elements on the partition $\mathcal{K}$.

Let $Q_r$ be space of polynomials in the variables $x_1, \ldots, x_d$, with real coefficients and of degree at most $r$ in each variable, i.e.,
\begin{equation*}
Q_r = \Big\{   \sum_{0\le \beta_1,\beta_2,\ldots,\beta_d\le r} c_{\beta_1 \beta_2\ldots \beta_d} x_1^{\beta_1} \cdots x_d^{\beta_d},
\quad \text{with}~~ c_{\beta_1 \beta_2\ldots \beta_d} \in\mathbb{R} \Big\}.
\end{equation*}
The $H^1$-conforming tensor-product finite element space, denoted by $S_h^r$, is defined as
\begin{equation*}
 S_h^r = \{ v\in H^1(\Omega): v|_K \in Q_r \,\, \text{for all} \,\, K\in \mathcal{K} \}.
\end{equation*}

We apply the Gauss--Lobatto quadrature in each subrectangle to approximate of the inner product, i.e.,
$$
(f,g)_h :=    \sum_{j_1=0}^{M_1 r}\cdots \sum_{j_d=0}^{M_d r} w_{j_1}\cdots w_{j_d} f(x_{ j_1},  \ldots, x_{j_d}) g(x_{ j_1},  \ldots, x_{j_d}) .
$$
This discrete inner product induces a norm
$$
\|f_h\|_h=\sqrt{(f_h,f_h)_h} \quad\forall\, f_h\in S_h^r.
$$

Similarly as the one-dimensional case, the spatially semi-discrete Gauss--Lobatto lumped mass method is to find $u_h\in S_h^r$ satisfying the following equation:
\begin{align}\label{md-semi-discrete-GL}
(\partial_t u_h , v_h)_h
+ (\nabla u_h, \nabla v_h) = (f, v_h)_h
\quad v_h\in S_h^r .
\end{align}
This can be furthermore written into a matrix-vector form:
\begin{align}
{\bf M}\dot {\bf u} - {\bf A} {\bf u}
={\bf M}{\bf f}
\end{align}
or equivalently
\begin{align}\label{md-semi-discrete-GL-matrix}
\dot {\bf u} - {\bf M}^{-1}{\bf A} {\bf u}
={\bf f}
\end{align}
where ${\bf u}$ is a vector consisting of the nodal values of $u_h$, and $\dot{\bf u}$ denotes the time derivative of the vector ${\bf u}$; ${\bf M}$ and ${\bf A}$ are the mass and stiffness matrices in the semidiscrete scheme \eqref{md-semi-discrete-GL}, respectively.

Similarly as the one-dimensional case, time discretization of \eqref{md-semi-discrete-GL-matrix} can be done by using the following exponential integrator:
\begin{align}\label{md-exp-GL}
\hat{\bf u}^n = e^{\tau {\bf M}^{-1}{\bf A}}{\bf u}^{n-1} + \int_{t_{n-1}}^{t_n}e^{(t_n-s) {\bf M}^{-1}{\bf A}} I_\tau^{(k-1)}{\bf f}(s)  \,ds .
\end{align}
where {   $I_\tau^{(k-1)}{\bf f}$, defined by \eqref{eqn:extrapolation}, is the extrapolated polynomial approximation of ${\bf f}$ on the time interval $[t_{n-1},t_n]$.}
Then we truncate the nodal vector $\hat{\bf u}^n$ by setting
\begin{equation}\label{md-truncation-GL}
\begin{split}
{\bf u}^n & =  \max \big(\hat{\bf u}^n, u_{\min} {\bf 1} \big) ,
\end{split}
\end{equation}
where ${\bf 1}$ denotes the vector with element $1$ in each component.

The error estimate for the multi-dimensional problem is the same as that presented in Theorem \ref{thm:exp-GL}. The details are omitted.



\section{The Allen--Cahn equation}\label{sec:AC}
In this section, we consider the cut-off exponential lumped mass method for the Allen--Cahn equation
\begin{equation}\label{eqn:AC}
    \left\{\begin{aligned} \partial_t  u  - \Delta u  &= f(u) &&  \mbox{in}\,\,\,\Omega\times(0,T],\\
\partial_nu &=0 && \mbox{on}\,\,\,\partial\Omega\times(0,T] , \\
    u|_{t=0} &= u_0 && \mbox{in}\,\,\,\in \Omega,
   \end{aligned} \right.
\end{equation}
where $f(u)=-F'(u)$ with a double-well potential $F$ that has two wells at $\pm\alpha$.
In this case, the solution of \eqref{eqn:AC}
satisfies the folowing {   maximum principle \cite{DuJuLiQiao:2020}}
\begin{align}\label{eqn:max-AC}
 \max_{x\in\overline\Omega} |u_0(x)|\le \alpha \quad \Rightarrow \quad |u(x,t)| \le \alpha,\qquad \text{for all}~~ (x,t) \in  \overline\Omega\times [0,T].
\end{align}
Two examples of such double-well potentials are given below.

\begin{example}[The Ginzburg--Landau free energy \cite{Allen-Cahn-1979}]\label{AC:1}
{\upshape
The Ginzburg--Landau double-well potential is
\begin{align}\label{eqn:F1}
F(u) = \frac1{4\epsilon^2}(1-u^2)^2 .
\end{align}
In this case, the right-hand side of \eqref{eqn:AC} is given by
\begin{align} \label{eqn:F1-f}
  f(u) = \frac{1}{\epsilon^2}(u-u^3) .
\end{align}
The solution of \eqref{eqn:AC}
satisfies the maximum principle \eqref{eqn:max-AC} with $\alpha=1$.
}
\end{example}

\begin{example}[The Flory--Huggins free energy \cite{Flory-1941,Huggins-1941}]\label{AC:2}
{\upshape
The logarithmic Flory--Huggins free energy is given by
\begin{align} \label{eqn:F2}
F(u) = \frac{\theta}{2\epsilon^2}[(1+u)\ln(1+u) + (1-u)\ln(1-u)] - \frac{\theta_c}{2\epsilon^2}u^2,
\end{align}
where $\theta$ and $\theta_c$ are two positive numbers satisfying $\theta<\theta_c$. In this case, the right-hand side of \eqref{eqn:AC} is given by
\begin{align} \label{eqn:F2-f}
  f(u) = -F'(u) =\frac{\theta_c}{\epsilon^2} u - \frac{\theta}{2\epsilon^2} \ln\Big(\frac{1+u}{1-u}\Big)   . 
\end{align}
Let $\pm\alpha$ be the two roots of $f(u)$, determined by
\begin{align} \label{eqn:beta}
  \frac{1}{2\alpha} \ln\bigg(\frac{1+\alpha}{1-\alpha}\bigg) = \frac{\theta_c}{\theta},
\quad\mbox{with}\,\,\,\alpha>0 .
\end{align}%
Then $\alpha\in(0,1)$ and the maximum principle \eqref{eqn:max-AC} holds.
}
\end{example}

In the next two subsections, we consider semi-discretization in time and full discretization, respectively.

\subsection{Semidiscrete multi-step exponential cut-off method}
The solution of \eqref{eqn:AC} satisfies
\begin{align}\label{eqn:sol-AC}
u(t_n) =  e^{\tau\Delta}  u(t_{n-1}) + \int_{t_{n-1}}^{t_n} e^{(t_n-s)\Delta} f(u(s)) \,ds ,
\end{align}
where $e^{t\Delta}$, $t\ge 0$, denotes the semigroup generated by the Laplacian operator, satisfying the following estimate
{    \cite[p. 117]{Larsson-Thomee-2003}}:
\begin{align}\label{contraction-e}
\|	e^{t\Delta} v\|_{L^2(\Omega)} \le \|v\|_{L^2(\Omega)} \quad\forall\, v\in L^2(\Omega).
\end{align}

{   Next, we introduce the time stepping scheme. The basic idea is the same as that of linear models in \eqref{exp-GL}.
On the time interval $[t_{n-1},t_n]$, we approximate $f(u(s))$ by the extrapolation polynomial
$$
 \sum_{j=1}^k L_j(s) f(u^{n-j}) .
$$
where $L_j(s)$ is the Lagrange basis polynomial defined in \eqref{eqn:extrapolation}.}
Correspondingly, for $n\ge k$ the semi-discrete extrapolated exponential integrator for \eqref{eqn:sol-AC} is given by
\begin{align}\label{AC-exp}
\hat u^n =  e^{\tau\Delta}  u^{n-1} + \int_{t_{n-1}}^{t_n} e^{(t_n-s)\Delta}
\sum_{j=1}^k L_j(s) f(u^{n-j})  \,ds ,
\end{align}
together with the truncation
\begin{align}\label{AC-trunc}
u^n =  \min(\max(\hat u^n,-\alpha),\alpha) .
\end{align}
Due to the truncation \eqref{AC-trunc}, the semi-discrete solution obtained from \eqref{AC-exp}-\eqref{AC-trunc} satisfies
$$
|u^n(x)|\le \alpha \quad \forall\, x\in \Omega .
$$
The accuracy of this truncated semi-discrete method is guaranteed by the following theorem.

\begin{theorem}\label{thm:err-AC}
Assume that $|u_0|\le \alpha$ and the maximum principle \eqref{eqn:max-AC} holds, and assume that the starting values $u^n$, $n=0,\dots,k-1$, are given and satisfies
$$
|u^n|\le \alpha \quad n=0,\dots,k-1.
$$
Then the semi-discrete solution given by \eqref{AC-exp}-\eqref{AC-trunc} satisfies
\begin{align}\label{AC-semi-discrete-max-pr}
|u^n|\le \alpha \quad n=k,\dots,N,
\end{align}
and
\begin{align}\label{eqn:err-AC}
  \| u(t_n) - u^n   \|_{L^2(\Omega)} \le  c \tau^k + c\sum_{n=0}^{k-1}\| u(t_n) - u^n   \|_{L^2(\Omega)} ,
\end{align}
provided that {   $f$ is locally Lipschitz continuous and $f(u) \in C^{k}([0,T];L^2\II)$.}
\end{theorem}
\begin{proof}
The cut-off operation \eqref{AC-trunc} guarantees \eqref{AC-semi-discrete-max-pr}. It suffices to prove the error estimate \eqref{eqn:err-AC}.

Let $e^n =  u(t_n) - u^n$. Then, since the cut-off operation is contractive, we have
\begin{align} \label{Ac-en-1}
 \|e^n \|_{L^2(\Omega)}
&= \|  \min(\max(u(t_n),-\alpha),\alpha)
 -\min(\max(\hat u^n,-\alpha),\alpha)  \|_{L^2(\Omega)} \notag \\
&\le   \|    u(t_n)  -  \hat u^n  \|_{L^2(\Omega)} .
\end{align}
Since $|u^n|\le\alpha$ and $f$ is locally Lipschitz continuous, it follows that
$$
\|f(u(t_{n-j})) - f(u^{n-j})\|_{L^2(\Omega)} \le c \|e^{n-j}\|_{L^2(\Omega)} .
$$
Hence, the difference between \eqref{eqn:sol-AC} and \eqref{AC-exp} yields
\begin{align}\label{AC-u-hat-u}
\|u(t_n)  -  \hat u^n\|_{L^2(\Omega)}
&=  \bigg\| e^{\tau\Delta}  e^{n-1}
+ \int_{t_{n-1}}^{t_n} e^{(t_n-s)\Delta}
[f(u(s))- \sum_{j=1}^k L_j(s) f(u(t_{n-j}))]\d s \notag \\
&\qquad
+\int_{t_{n-1}}^{t_n} e^{(t_n-s)\Delta}  \sum_{j=1}^k L_j(s) [f(u(t_{n-j})) - f(u^{n-j})]\,ds \bigg\|_{L^2(\Omega)} \notag \\
&\le
\|e^{n-1} \|_{L^2(\Omega)}
+\int_{t_{n-1}}^{t_n} c\tau^k \d s
+\int_{t_{n-1}}^{t_n} c \sum_{j=1}^k \|e^{n-j}\| _{L^2(\Omega)} \d s \notag \\
&\le
\|e^{n-1} \|_{L^2(\Omega)}
+c\tau^{k+1}
+c\tau \sum_{j=1}^k \|e^{n-j}\|_{L^2(\Omega)} ,
\end{align}
where we have used \eqref{contraction-e} in deriving the first inequality of \eqref{AC-u-hat-u}.
Combining \eqref{Ac-en-1} and \eqref{AC-u-hat-u}, we have
\begin{align}
\|e^n\|_{L^2(\Omega)}
&\le
\|e^{n-1} \|_{L^2(\Omega)}
+c\tau^{k+1}
+c\tau \sum_{j=1}^k \|e^{n-j}\|_{L^2(\Omega)} .
\end{align}
Then, summing up the inequality above for $n=k,\dots,m$, we obtain
\begin{align*}
\|e^m\|_{L^2(\Omega)}
&\le
c\sum_{n=0}^{k-1}\|e^{n}\|_{L^2(\Omega)}
+c\tau^{k}
+c\tau \sum_{n=k}^m\|e^{n-1}\|_{L^2(\Omega)} .
\end{align*}
By applying Gronwall's inequality, we obtain the desired estimate in Theorem \ref{thm:err-AC}.
\hfill\end{proof}

\subsection{Exponential cut-off methods}\label{section:fully-AC}

For a given $u^{n-1}_h\in S_h^r$, the nodal vector ${\bf u}^{n-1}$ is uniquely determined. We apply the exponential lumped mass method \eqref{exp-GL} to the Allen--Cahn equation \eqref{eqn:AC}:
\begin{align}\label{exp-GL-AC}
\hat{\bf u}^n = e^{\tau {\bf M}^{-1}{\bf A}}{\bf u}^{n-1} + \int_{t_{n-1}}^{t_n}e^{(t_n-s) {\bf M}^{-1}{\bf A}} \sum_{j=1}^k L_j(s) f({\bf u}^{n-j})  \,ds ,
\end{align}
where the nonlinear term is extrapolated as the semi-discrete scheme \eqref{AC-exp}. Then we truncate the nodal vector $\hat{\bf u}^n$ by
\begin{equation}\label{truncation-AC}
\begin{split}
{\bf u}^n & =  \min\big(\max \big(\hat{\bf u}^n, -\bm{\alpha} \big) ,\bm{\alpha}\big) .
\end{split}
\end{equation}
Similarly as the linear parabolic problem, the finite element function corresponding to the nodal vector $\hat{\bf u}^n$ coincides with the solution $\hat u_h(t_n)$ obtained from the initial-value problem for $t\in(t_{n-1},t_n]$
\begin{align}\label{fully-discrete-AC}
\left\{\begin{aligned}
&(\partial_t \hat u_h , v_h)_h
+ (\partial_x \hat u_h, \partial_x v_h)
= \Big(\sum_{j=1}^k L_j  f(u_h^{n-j}) , v_h \Big)_h
\quad v_h\in S_h^r, \\
&\hat u_h(t_{n-1})=u^{n-1}_h .
\end{aligned}\right.
\end{align}
The nodal vector obtained from \eqref{truncation-AC} can be used to construct a piecewise polynomial $u^{n}_h\in S_h^r$.

The exact solution of \eqref{eqn:AC} satisfies
\begin{align}\label{weak-AC}
(\partial_t u , v) + (\partial_x u , \partial_x v) = (f(u), v) \quad \forall\, v\in H^1\II,
\end{align}
which implies
\begin{align}\label{weak-AC-fe}
(\partial_t u, v_h)_h + (\partial_x \Pi_h u, \partial_x v_h)  = \Big(\sum_{j=1}^k L_j f(u(t_{n-j})), v_h\Big)_h + \mathcal{E}(v_h) \quad\forall\, v_h\in S_h^r,
\end{align}
where $\mathcal{E}(v_h)$ denotes the error of quadrature, interpolation and extrapolation, satisfying the following estimate (similarly as the linear problem):
\begin{align}\label{AC-quadrature-error}
|\mathcal{E}(v_h)|
&\le
c\sum_{i=1}^{M}\|\partial_t u \, v_h\|_{W^{2r,1}(I_i)}h^{2r}
+c\sum_{i=1}^{M}\|f(u) v_h\|_{W^{2r,1}(I_i)}h^{2r} \notag \\
&\quad\,
+\Big|\Big(f(u)-\sum_{j=1}^k L_j(s) f(u(t_{n-j})),v_h\Big)_h\Big| + c|(\partial_x(u-\Pi_hu),\partial_xv_h)|
\notag\\
&\le
ch^{2r} \sum_{i=1}^{M}\sum_{j=0}^r\|\partial_t u\|_{H^{2r-j}(I_i)} \|v_h\|_{H^j(I_i)}
+ch^{2r} \sum_{i=1}^{M}\sum_{j=0}^r\|f(u)\|_{H^{2r-j}(I_i)} \|v_h\|_{H^j(I_i)}   \notag\\
&\quad\,
+ c\tau^k \|v_h\|_{h}
+ch^{r}\|u\|_{H^{r+1}(I_i)}\|\partial_x v_h\|_{L^2(\Omega)}
\notag \\
&\le
ch^{2r} \sum_{i=1}^{M}(\|\partial_t u\|_{H^{2r}(I_i)}+\|f\|_{H^{2r}(I_i)}) \|v_h\|_{H^r(I_i)} \notag\\
&\quad\,
+ c\tau^k \|v_h\|_{h}
+ch^{r}\|u\|_{H^{r+1}(I_i)}\|\partial_x v_h\|_{L^2(\Omega)}
\notag \\
&\le
c(\tau^k+h^r) (\|v_h\|_h +\|\partial_xv_h\|_{L^2(\Omega)} )
\qquad \forall\, v_h\in S_h^r ,
\end{align}
where the last inequality follows from the inverse inequality of the finite element space.

The accuracy of the fully discrete scheme \eqref{exp-GL-AC}-\eqref{truncation-AC} for the Allen--Cahn equation is presented in the following theorem.

\begin{theorem}\label{thm:fully-discrete-AC}
Assume that $|u_0|\le \alpha$ and the maximum principle \eqref{eqn:max-AC} holds, and assume that the starting values $u_h^n$, $n=0,\dots,k-1$, are given and satisfies
$$
|u_h^n(x_j)|\le \alpha , \quad j=0,\dots,Mr,\quad n=0,\dots,k-1.
$$
Then the   {fully discrete} solution given by \eqref{AC-exp}-\eqref{AC-trunc} satisfies
\begin{align}\label{AC-fully-discrete-max-pr}
|u_h^n(x_j)|\le \alpha , \quad j=0,\dots,Mr,\quad n=k,\dots,N,
\end{align}
and  {for $n=k,\ldots,N$}
\begin{align}\label{fully-err-AC}
  \| u(t_n) - u_h^n \|_{L^2(\Omega)} \le  c (\tau^k + h^r) + c\sum_{n=0}^{k-1}\| u(t_n) - u_h^n \|_{L^2(\Omega)} ,
\end{align}
provided that {   $u\in C^1([0,T];H^{2r}(\Omega))$, $f$ is locally Lipschitz continuous and $f(u)\in C^k([0,T];L^2\II)\cap C([0,T];H^{2r}\II)$.}
\end{theorem}

\begin{proof}
The cut-off operation \eqref{AC-trunc} guarantees \eqref{AC-fully-discrete-max-pr}. It suffices to prove the error estimate.
 {To this end, we define $e_h^n = \Pi_h u(t_n) - u_h^n$
and note that each node $x_j$ there holds
\begin{align*}
|e_h^n(x_j)|
&= | \min\big( \max \big( u(x_j,t_n) , - \alpha ) ,  \alpha  \big)
- \min\big( \max \big(\hat u_h(x_j, t_n), -\alpha) ,  \alpha \big)  |\\
&\le |  u(x_j,t_n)  -  \hat u_h(x_j, t_n)  | ,
\end{align*}
where the piecewise-defined function $\hat u_h$ is given in \eqref{fully-discrete-AC}.
Then for $\hat e_h=\Pi_hu-\hat u_h$ we have
\begin{align}\label{bfe-AC}
\| e_h^n  \|_h
\le  \|  \hat e_h(t_n)  \|_h\quad\mbox{and}\quad
\| e_h^{n-1} \|_h = \|\hat e_h(t_{n-1})\|_h.
\end{align}}

Besides, the difference between \eqref{weak-AC-fe} and \eqref{fully-discrete-AC} yields the error equation
{    for $t\in(t_{n-1},t_n]$}
\begin{align}\label{error-eq-GL-AC}
\left\{\begin{aligned}
&(\partial_t \hat e_h , v_h)_h
+ (\partial_x \hat e_h, \partial_x v_h) =
\bigg( \sum_{j=1}^k L_j [f(u(t_{n-j}))-f(u_h^{n-j})], v_h \bigg)_h
+\mathcal{E}(v_h) , \\
&\hat e_h(t_{n-1})= \Pi_hu(t_{n-1})-u^{n-1}_h ,
\end{aligned}\right.
\end{align}
which holds for all $v_h\in S_h^r$.

Substituting $v_h=\hat e_h$ into the equation above, we obtain
\begin{align*}
& \frac{\d}{\d t}\bigg( \frac12 \|\hat e_h\|_h^2  \bigg)
+ \|\partial_x \hat e_h\|_{L^2(\Omega)}^2 \\
&\le
c\sum_{j=1}^k\|f(u(t_{n-j}))-f(u_h^{n-j})\|_h \|\hat e_h\|_h
+\mathcal{E}(\hat e_h)  \\
&\le
c\sum_{j=1}^k\|e_h^{n-j}\|_h \|\hat e_h\|_h + c(\tau^k+h^r) (\|\hat e_h\|_h +\|\partial_x\hat e_h\|_{L^2(\Omega)} )  \\
&{   \le
c(\tau^k + h^r)^2 + c\sum_{j=1}^k\|e_h^{n-j}\|_h^2
+ \frac12\|\hat e_h\|_h^2 + \frac12\|\partial_x\hat e_h\|_{L^2(\Omega)}^2 },
\end{align*}
which furthermore reduces to
\begin{align*}
\frac{\d}{\d t}\bigg( \frac12 \|\hat e_h\|_h^2  \bigg)
&\le
c(\tau^k + h^r)^2 + c\sum_{j=1}^k\|e_h^{n-j}\|_h^2
+ \frac12\|\hat e_h\|_h^2  .
\end{align*}
By applying Gronwall's inequality, we obtain
\begin{align*}
\|\hat e_h(t_n)\|_h^2
&\le e^{\tau} c\tau (\tau^k + h^r)^2
+ e^{\tau} c\tau \sum_{j=1}^k\|e_h^{n-j}\|_h^2
+ e^{\tau} \|\hat e_h(t_{n-1})\|_h^2 \\
&\le c\tau (\tau^k + h^r)^2
+ c\tau \sum_{j=1}^k\|e_h^{n-j}\|_h^2
+ (1+c\tau) \|\hat e_h(t_{n-1})\|_h^2.
\end{align*}
 {Then, 
using \eqref{bfe-AC}, we have
\begin{align*}
\|e_h^n\|_h^2
&\le c\tau (\tau^k + h^r)^2
+ c\tau \sum_{j=1}^k\|e_h^{n-j}\|_h^2
+ (1+c\tau) \|e_h^{n-1} \|_h^2 ,
\end{align*}
which can be rewritten as
\begin{align*}
\frac{ \|e_h^n \|_h^2 - \|e_h^{n-1} \|_h^2 }{\tau}
&\le c (\tau^k + h^r)^2
+   c \sum_{j=0}^k\|e_h^{n-j}\|_h^2 .
\end{align*}
By using discrete Gronwall's inequality, we obtain
\begin{align*}
\max_{1\le n\le N}\|e_h^n\|_h^2
&\le e^{cT}  c(\tau^k + h^r)^2 + e^{cT} c \sum_{n=0}^{k-1}\|e_h^n\|_h^2 .
\end{align*}}
This and the norm equivalence in Lemma \ref{lem:norm-equivalence} imply the desired result in Theorem \ref{thm:fully-discrete-AC}.
\hfill\end{proof}

{   
\begin{remark}
{\upshape
The starting values at the first $k-1$ time levels can be first computed by a single-step high-order implicit method (such as the Runge--Kutta method) and then post-processed by the cut-off operation.
The accuracy will not be destroyed by the cut-off operation since there are only $k-1$ time levels (without accumulation in time).
In this way, the obtained starting values have the right order in time, and in particular satisfying $|u^n_h(x_j)| \le \alpha$.
}
\end{remark}
}

\section{Numerical results}\label{section:numerical}
In this section, we present numerical results to illustrate
the fully discrete scheme \eqref{exp-GL-AC}-\eqref{truncation-AC} with one- and two-dimensional examples.
In our computation,
we compute the exponential integator $e^{\Delta t}$ by using inverse Laplace
transform and approximating an contour integral over a hyperbola (see e.g., \cite[Section 4]{WeidemanTrefethen:2007}
and \cite[Section 4]{JinLSZ:2016}).

\begin{example}[The Ginzburg--Landau potential]\label{Example1}
{\upshape We begin with the following one-dimensional Allen--Cahn equation:
\begin{equation}\label{eqn:AC-1}
    \left\{\begin{aligned} \partial_t  u  - \epsilon^2\partial_{xx} u  &=  f(u),&&
    \mbox{in}\,\,\,\Omega\times(0,T],\\
\partial_xu &=0,&&
\mbox{on}\,\,\,\partial\Omega\times(0,T] , \\
    u|_{t=0} &= u_0 , && \mbox{in}\,\,\,\Omega,
   \end{aligned} \right.
\end{equation}
where $\Omega=(-1,1)$ and $\epsilon=0.01$, and  $f(u) =  u-u^3 $  is the Ginzburg--Landau double-well potential in Example \ref{AC:1}. The solution of \eqref{eqn:AC-1}
satisfies the maximum principle \eqref{eqn:max-AC} with $\alpha=1$.
The initial value is given by
\begin{equation}\label{ini-01}
u_0(x) = \alpha\Big[\chi_{(-1,-\frac12)}(x) + \chi_{(-\frac12,1)} (x) \cos\Big(\frac32\pi \big(x+\frac12\big)\Big)\Big] ,
\end{equation}
where  $\chi$ denotes the characteristic function. This initial value is chosen to be smooth and satisfying the Neumann boundary condition.

We solve \eqref{eqn:AC-1} by the proposed method with temporal stepsize $\tau=T/N$ and spatial mesh size $h=2/M$. For a $k$-step method with $k=2,3,4$, we compute the numerical solution at the first $k-1$ time levels by using the two-stage Gauss--Legendre Runge--Kutta method (cf. \cite[p. 47]{Iserles:1996}), which has 4th-order accuracy in time, and therefore sufficiently accurate as starting approximations in view of the error estimate in \eqref{fully-err-AC}. Cutting off the numerical solutions at the first $k-1$ time levels does not affect the accuracy.

{   We present the temporal error $e_\tau$ and spatial error $e_h$ in in Tables \ref{tab:a-time} and \ref{tab:a-space}, respectively.
Since the exact solution is unavailable, we compute the reference solution by the exponential cut-off scheme with $r=k=4$ and a finer mesh.
In particular, the temporal error $e_\tau$ is computed by fixing the spatial mesh size $h=1/400$
and comparing the numerical solution with a reference solution (corresponding to $\tau = 1/400$).
Similarly, the spatial error $e_h$ is computed to by fixing the temporal step size $\tau=1/400$ and
comparing the numerical solutions with a reference solution (corresponding to $h = 1/400$).  }

\begin{table}[htb!]
\caption{Example {\rm\ref{Example1}:} Temporal error $e_\tau$
at $T=1$ and $T=5$, with  $\tau = T/N$ and $h=1/400$.}\label{tab:a-time}
\begin{center}
     \begin{tabular}{|c|c|ccccc|c|}
     \hline
     $k$ & $ N$    &$10 $ &$20 $ & $40 $ & $80 $ &$160 $& rate \\
     \hline
        2  &   $T=1$       &7.29e-4 &1.88e-4 &4.79e-5 &1.21e-5 &3.03e-6  &  $\approx$ 1.99 \\
      &  $T=5$       &1.75e-3 &2.81e-4 &6.87e-5 &1.70e-5 &4.22e-6  &  $\approx$ 2.01 \\
      \hline
   3  &   $T=1$       &1.15e-4 &1.52e-5 &1.95e-6 &2.46e-7 &3.10e-8  &  $\approx$ 2.99 \\
      &  $T=5$       &1.18e-2 &8.16e-5 &9.01e-6 &1.07e-6 &1.31e-7  &  $\approx$ 3.00 \\
      \hline
     4 & $T=1$       &2.46e-5 &1.64e-6 &1.04e-7 &6.50e-9 &3.96e-10  &  $\approx$ 4.02 \\
      &  $T=5$       &1.22e-1 &3.65e-3 &2.30e-6 &1.38e-7 &8.25e-9  &  $\approx$ 4.09 \\
      \hline
     ETD-RK2 \cite{DuJuLiQiao:2019}& $T=1$       &2.69e-3 &7.32e-4 &1.91e-4 &4..89e-5 &1.23e-5  &  $\approx$ 1.98 \\
      &  $T=5$       &1.30e-2 &2.91e-3 &7.29e-4 &1.74e-4 &3.24e-5  &  $\approx$ 2.24 \\
      \hline
     \end{tabular}
\end{center}
\end{table}

\begin{table}[htb!]
\caption{Example {\rm\ref{Example1}:} Spatial error $e_h$
at $T= 1$,  with $h=2/M$ and $\tau=1/400$.}\label{tab:a-space}
\begin{center}
     \begin{tabular}{|c|ccccc|c|}
     \hline
     $r\backslash M$    &$10 $ &$20 $ & $40 $ & $80$ & $160$  & rate \\
     \hline
         1      &4.59e-2 &1.22e-2 &3.14e-3 &7.90e-5 &7.90e-5   &  $\approx$ 1.97  \\

        2        &6.18e-3 &9.22e-4 &1.13e-4 &1.42e-5   &1.76e-6 &  $\approx$ 3.00 \\

     3         &1.27e-3 &6.49e-5 &5.43e-6 &3.44e-7 &2.15e-8  &  $\approx$ 3.99 \\

     4       &1.61e-4 &1.19e-5 &2.96e-7 &9.36e-9  &3.14e-10  &  $\approx$ 4.95 \\

  ETD-RK2   \cite{DuJuLiQiao:2019} &6.02e-2 &3.30e-2 &8.03e-3 &2.03e-3  &5.08e-4  &  $\approx$ 2.00  \\
      \hline
     \end{tabular}
\end{center}
\end{table}

Numerical results show that the temporal discretization error $e_\tau$ is $O(\tau^k)$, which is consistent with the theoretical result proved in Theorem \ref{thm:fully-discrete-AC}.
Numerical results show that the spatial error is $O(h^{r+1})$, one order higher than the result proved in Theorem
\ref{thm:fully-discrete-AC}.
For comparison, in Tables \ref{tab:a-time} and \ref{tab:a-space} we also presented the errors of the stabilized ETD-RK2 method with stabilization parameter $\kappa=2$, which has been proved to satisfy the maximum principle \cite{DuJuLiQiao:2019}.

\begin{figure}[h]
\centering
\includegraphics[trim = .1cm .1cm .1cm .1cm, clip=true,width=1\textwidth]{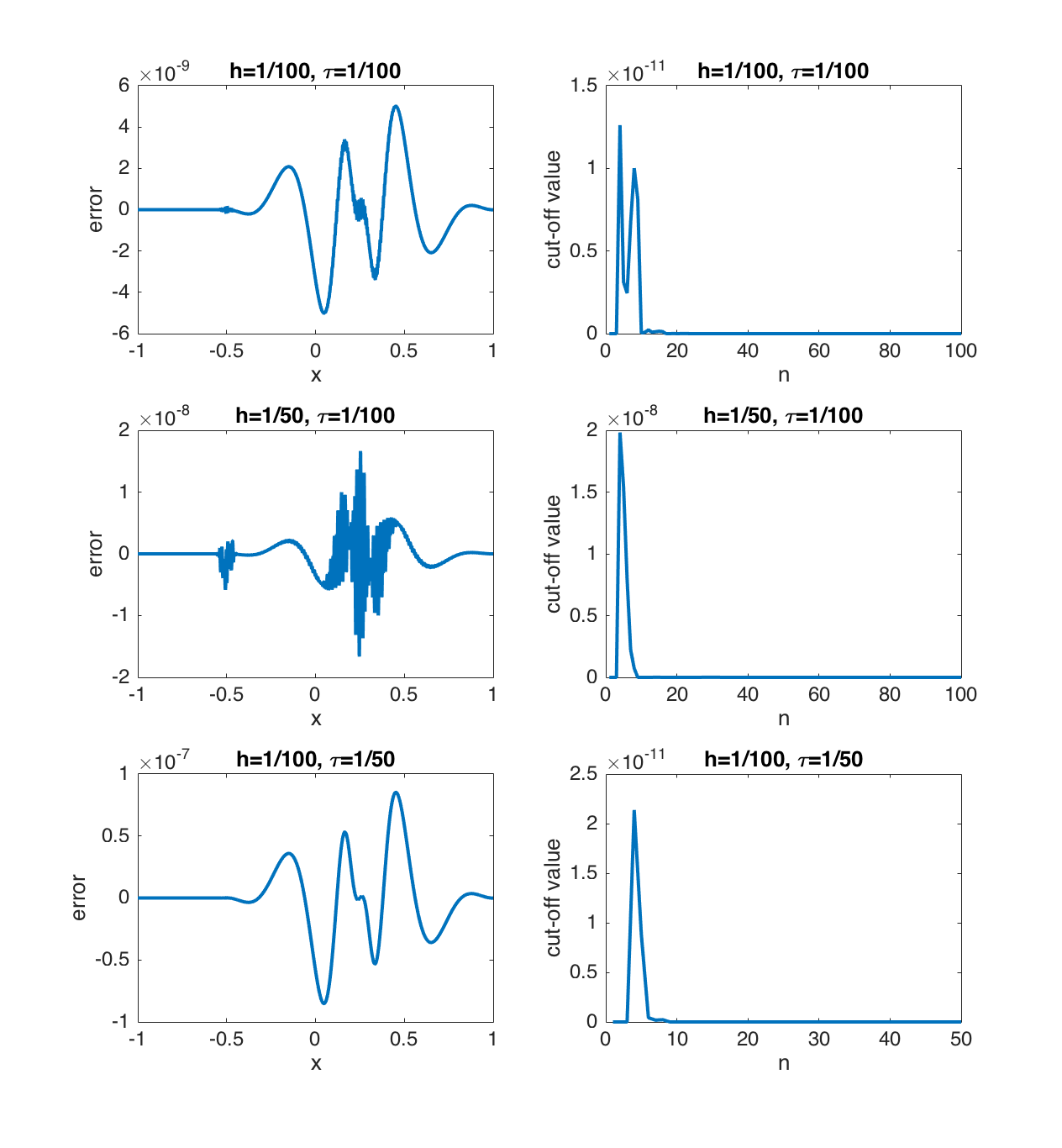}
\vspace{-35pt}
\caption{Example {\rm\ref{Example1}:} Error at $T=1$ and maximal cut-off value at each time level.
\label{fig:err_trun}}
\end{figure}

In Figure \ref{fig:err_trun}, we plot the error of the numerical solution $e(x)=u_h^N(x) - u(x,T)$
and the maximal cut-off value
$$\rho^n =  \max |u^n - \hat u^n |.$$
Numerical results show that the cut-off function is active in the computation,
especially in the starting stage. Meanwhile, it is not surprising that a coarse mesh size will result in a large cut-off value, even though it does not affect the convergence rate, cf. Tables \ref{tab:a-time} and \ref{tab:a-space}.

}
\end{example}

\vskip5pt
\begin{example}[The Flory--Huggins potential]\label{Example2}
{\upshape
Now we consider the one-dimensional Allen--Cahn model \eqref{eqn:AC-1}--\eqref{ini-01} with the logarithmic Flory--Huggins potential and
\begin{equation}\label{FH}
 f(u)  = u -  \frac{1}{8} \ln \Big(\frac{1+u}{1-u}\Big).
\end{equation}
The exact solution satisfies the maximum principle $|u|\le \alpha$ with $\alpha=0.99933$ being the root of the equation
\begin{align*} 
   \ln\Big(\frac{1+\alpha}{1-\alpha}\Big) = 8\alpha,
\end{align*}%
which can be approximately solved by Newton's method.

In Figure \ref{fig:b-without-cut}, we plot the numerical solution without cut-off post-processing, at $T=1$, where $h=\tau=1/100$. The numerical solution becomes complex, and its real part significantly exceeds $[-1,1]$.
{This shows that the scheme without cut-off operation is unstable and inaccurate.
The numerical solutions with cut-off operation are plotted in Figure  \ref{fig:b-cut}, where we see that the numerical solution is closed to the exact solution} (a reference solution with
$k=r=4$ and $h=\tau=1/1000$). We also see that the cut-off values do not decay but keep stable for large $n$.

The error and convergence rate are presented in Table \ref{tab:b-space}, where we also compare the proposed high-order scheme with the stabilized ETD-RK2 method \cite{DuJuLiQiao:2019}.
In the stabilized ETD-RK2 method, the stabilisation parameter $\kappa$ should satisfy the following criterion \cite{ShenTangYang:2016}:
\begin{equation}\label{stab-cond}
 \frac{1}{\tau} + \kappa \ge \frac{1}{4(1-\alpha^2)} - 1.
\end{equation}
Since $\alpha=0.99933$ is too close to $1$ (at which the logarithmic potential is singular), the accuracy of numerical solution is affected and the convergence rates in \ref{tab:b-space} is not as perfect as in Tables \ref{tab:a-time}--\ref{tab:a-space}.  Neverthness, the numerical results in Table \ref{tab:b-space} still show the superiority of the cut-off exponential lumped mass FEM proposed in this article.

\begin{figure}[h]
\centering
\includegraphics[trim = .1cm .1cm .1cm .1cm, clip=true, width=.9\textwidth]{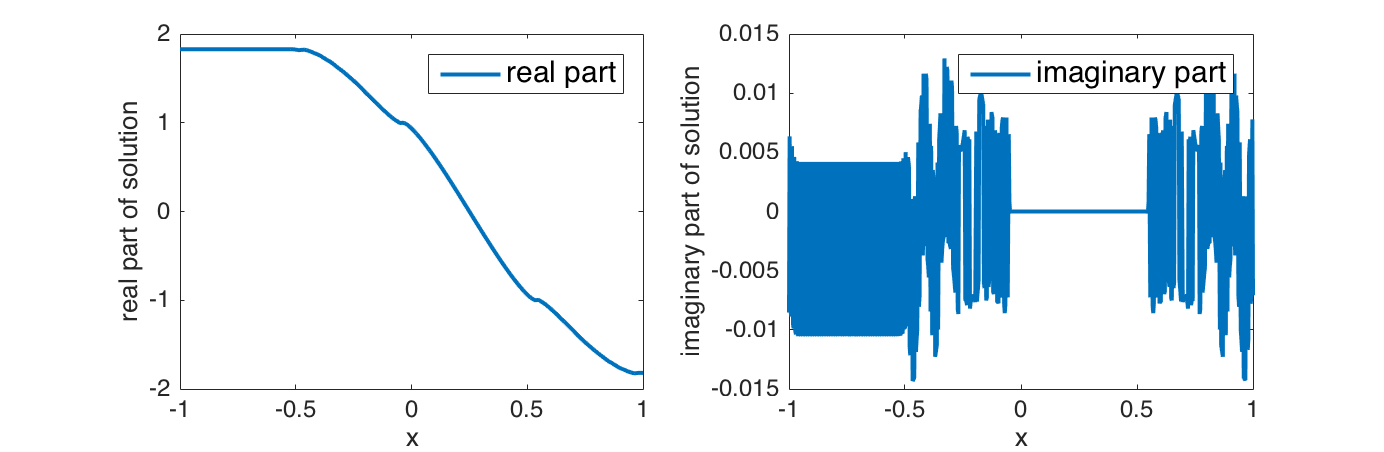}
\vspace{-10pt}
\caption{{\rm (Example \ref{Example2})} Numerical solution without cut-off operations, with $h=\tau=1/100$.
\label{fig:b-without-cut}}
\vspace{10pt}
\centering
\includegraphics[trim = .1cm .1cm .1cm .1cm, clip=true,width=.9\textwidth]{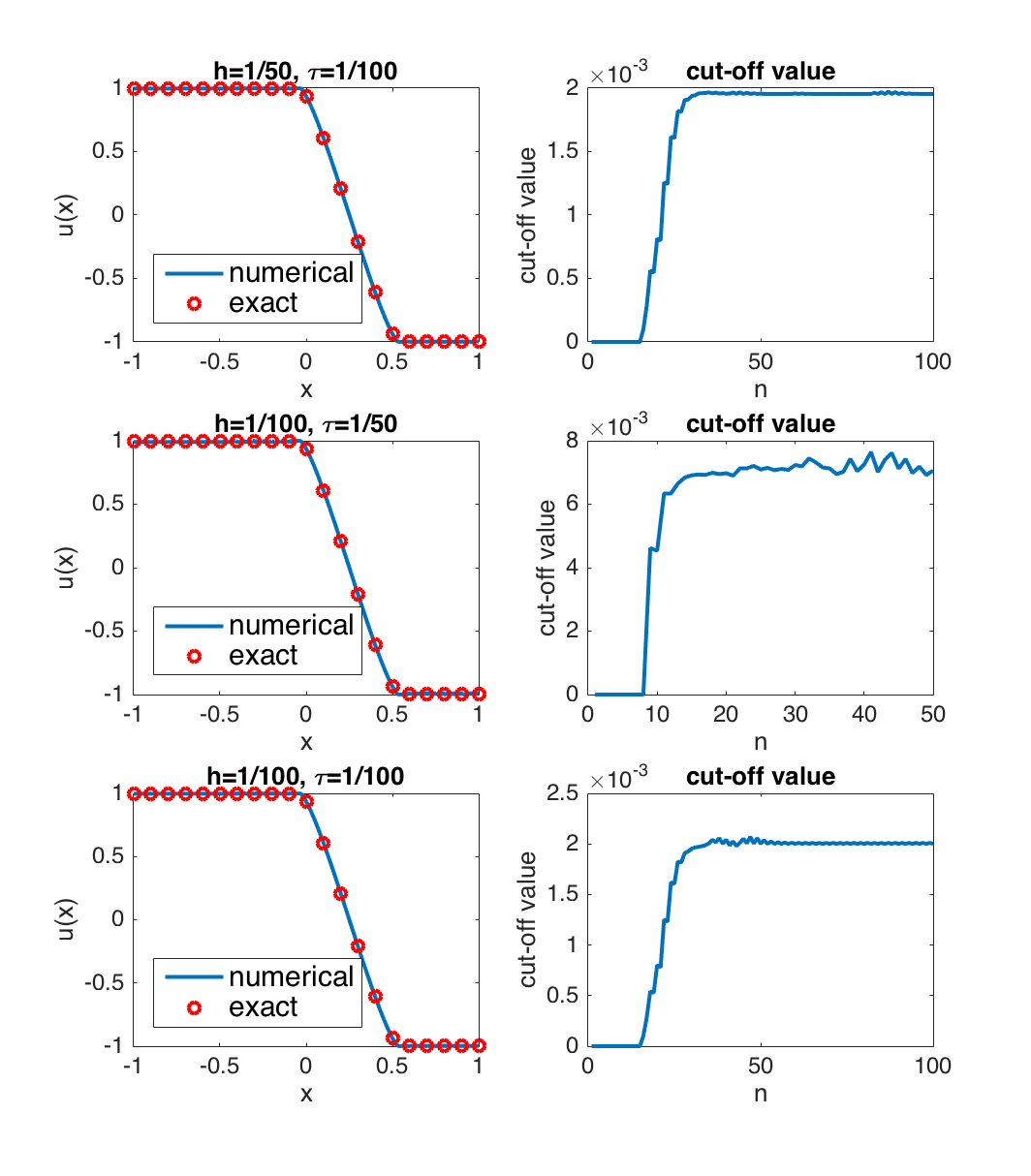}
\vspace{-35pt}
\caption{{\rm (Example \ref{Example2})} Numerical solution $T=1$ and cut-off value at each time level.
\label{fig:b-cut}}
\end{figure}

\begin{figure}[h]
\centering
\includegraphics[trim = .1cm .1cm .1cm .1cm, clip=true,width=0.95\textwidth]{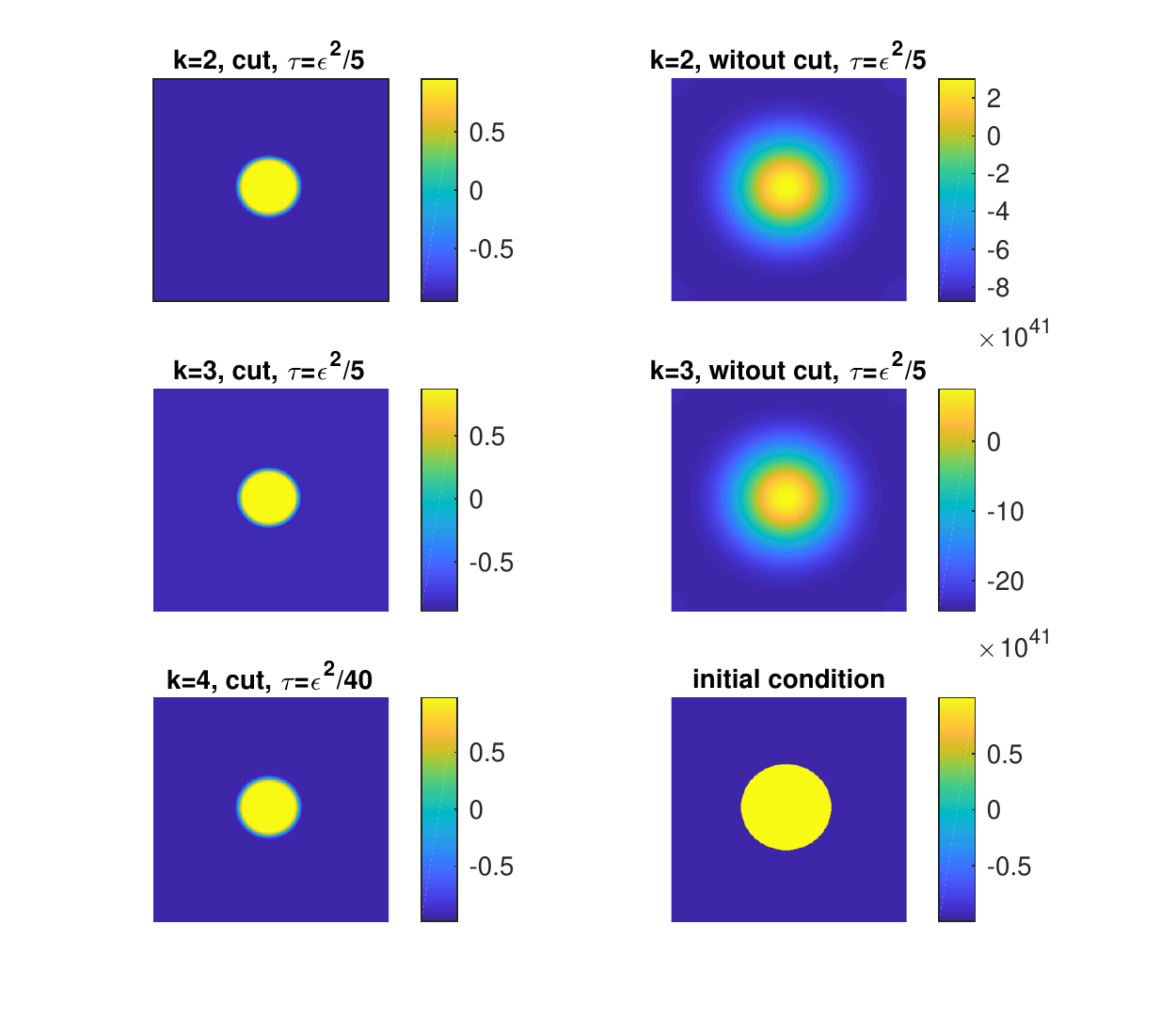}\vskip-0.5in
\caption{{\rm(Example \ref{Example3})} Numerical solution with or without cut-off post-processing, at $T=\epsilon$.
\label{fig:c-cut}}
\vspace{15pt}
\centering
\includegraphics[trim = .1cm .1cm .1cm .1cm, clip=true,width=0.75\textwidth]{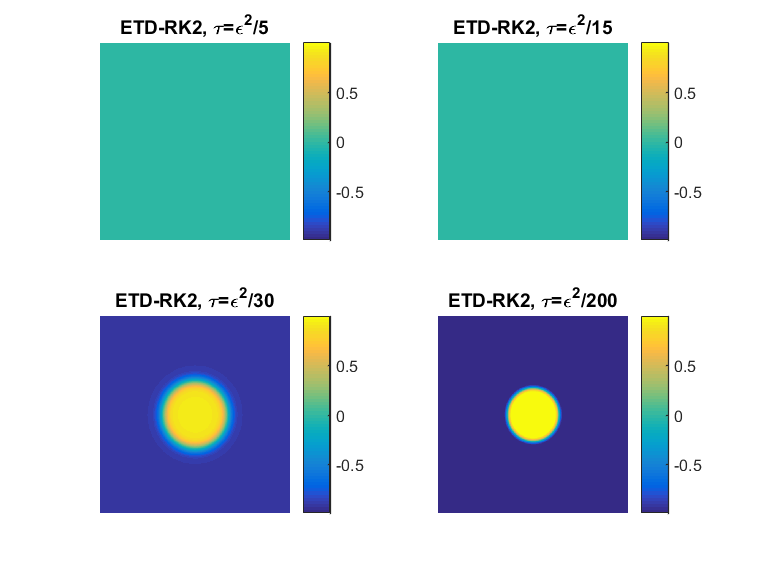}\vskip-0.35in
\caption{{\rm(Example \ref{Example3})} Numerical solution of stabilized ETD-RK2 scheme  \cite{DuJuLiQiao:2019}, at $T=\epsilon$.
\label{fig:c-etdrk2}}
\end{figure}

\begin{table}[h]
\caption{{Example \ref{Example2}:} Error of numerical solutions at $T = 1$, with $h=2/M$ and $\tau=1/M$.}\label{tab:b-space}
\begin{center}
     \begin{tabular}{|c|ccccc|c|}
     \hline
     $r,k\backslash M$    &$20 $ &$40 $ & $80 $ & $160$ & $320$   \\
     \hline

     $r=k=2$         &1.12e-2 &4.03e-3 &4.70e-4 &1.83e-4   &1.27e-5   \\
                       &             &   1.47      & 3.10   &  1.36  &  3.85\\
                         \hline
    $r=k=3$         &2.85e-2 &1.11e-2 &3.91e-3 &5.33e-4 &4.58e-6  \\
     &             &   1.36     & 1.50   &  2.88  &  6.86\\
       \hline
  ETD-RK2 (cf. \cite{DuJuLiQiao:2019}) &2.11e-1 &1.25e-1 &5.60e-2 &1.97e-2  &5.92e-3    \\
    &             &   0.75    & 1.16   &  1.51  &  1.73\\
      \hline
     \end{tabular}
\end{center}
\end{table}

}

\end{example}

\begin{example}[The two-dimensional Allen--Cahn equation]\label{Example3}
{\upshape
We consider the following Allen--Cahn equation in the two-dimensional domain $\Omega=(0,2\pi)^2$:
 \begin{equation}\label{eqn:AC-2D}
    \left\{\begin{aligned} \partial_t  u  -  \Delta u  &= f(u),&&
    \mbox{in}\,\,\,\Omega\times(0,T],\\
    u|_{t=0} &= u_0 , && \mbox{in}\,\,\,\Omega,\\
    \text{$u$ satisfies the} & ~\text{periodic boundary condition,}
   \end{aligned} \right.
\end{equation}
with the Flory--Huggins potential 
$$
f(u)  = \frac1{\epsilon^2} \bigg[u -  \frac{1}{8} \ln \Big(\frac{1+u}{1-u}\Big)\bigg] 
$$
and a parameter $\epsilon = 0.01$. The initial condition is given by
$$ u(x_1,x_2,0) = \alpha \Big(\chi\big((x_1-\pi)^2 + (x_2-\pi)^2 \le 1.2\big)-2\Big),
$$
which yields a circular initial interface.

In the spatial discretization, we divide the interval $(0,2\pi)$ into $M$ sub-intervals, with mesh size $h=2\pi/M$. Correspondingly, the domain $\Omega$ is divided into $M^2$ small squares. We apply the lumped mass FEM in space with $r=1$ and $M=1/500$, and investigate the numerical results given by different time-stepping schemes. The exponential integrator is evaluated by using FFT.

In Figure \ref{fig:c-cut}, we plot the numerical solution of the proposed cut-off expontential scheme \eqref{exp-GL-AC}--\eqref{truncation-AC} at $T=1$. Since the closed form of the exact solution is unavailable, we compute the reference solution by choosing $k=4$ and a sufficiently small stepsize $\tau=1/2000$. From Figure \ref{fig:c-cut} we see that the numerical solution given by the proposed method with stepsize $\tau=1/500$ agrees well with the reference solution, while the numerical method without cut-off operation is inaccurate and unstable. These numerical results show the effectiveness of the cut-off operation in improving the accuracy of numerical solutions (without restriction on the temporal stepsize and spatial mesh size).

In Figure \ref{fig:c-etdrk2}, we plot the numerical solutions given by the stabilized ETD-RK2 method \cite{DuJuLiQiao:2019} for comparison with the method proposed in this article. In the ETD-RK2 method, the stabilization parameter is chosen to be (cf. \cite{ShenTangYang:2016,DuJuLiQiao:2019})
$$\kappa = \frac1{\epsilon^2}\Big[\frac1{4(1-\alpha^2)} - 1\Big]\quad \text{with}\quad \alpha\approx 0.99933.$$
in order to preserve the maximum principle in the discrete level. For a sufficiently small stepsize, such as $\tau=\epsilon^2/200$,
 the ETD-RK2 yields the same pattern as our method. For larger time stepsizes, such as 
$\tau=\epsilon^2/5,\epsilon^2/15$ and $\epsilon^2/30$, the ETD-RK2 method does not yield the correct pattern, 
 while our method still yields the correct pattern (as shown in Figure \ref{fig:c-cut}). 
 The reason is that the stabilization term is very large when the stepsize $\tau$ is not small, 
 and this significantly affects the accuracy of numerical solutions. Since our method does 
 not contain any stabilizer (which introduces extra artificial error), it is more robust and accurate for larger time stepsizes.

}
\end{example}

\section{Conclusion}

We have proposed a class of arbitrarily high-order MMP methods for semilinear parabolic equations based on a $k$-step exponential integrator in time and $r$th-order lumped mass finite element methods in space, and a cut-off operation which eliminates the extra values violating the maximum principle at nodal points of the finite elements. We have proved that the proposed method has at least $k$th-order convergence in time and $r$th-order convergence in space, without restriction on the temporal stepsize and spatial mesh size. The numerical results in Example \ref{Example3} have shown the accuracy and effectiveness of the proposed method for the Allen--Cahn equation in capturing a sharp interface with relatively large time stepsize. The numerical results in Examples \ref{Example1}--\ref{Example2} show that the solution has $k$th-order convergence in time and $(r+1)$th-order convergence in space, one-order higher than the result proved in this paper. The loss of one-order convergence in our proof is due to the cut-off operation. Theoretical proof of the $(r+1)$th-order convergence in space is still challenging.

{   
It is known that \eqref{exp-GL-AC} is the exponential multistep method in \cite[Section 2.5]{Hochbruck-Ostermann-2010} for semilinear parabolic problems. In addition to the exponential multistep method, one can also replace \eqref{exp-GL-AC} by the exponential Runge--Kutta method  (cf. \cite[Section 2.3]{Hochbruck-Ostermann-2010}), e.g., the exponential Euler method and higher-stage methods \cite{Cox-Matthews-2002,Hochbruck-Ostermann-2005a}.
Practically, one can cut the nodal values of the numerical solutions using any exponential integrator. We have focused on the exponential multistep method in this article because it is linearly implicit, therefore relatively easier for both implementation and error analysis. The convergence and error estimates for the exponential Runge--Kutta cut-off method (with internal stages) still remains open.
}

\bigskip
\section*{Acknowledgements}
The research of B. Li is partially supported by a Hong Kong RGC grant (project no. 15300519). The work of J. Yang is supported by National Natural Science Foundation of China (NSFC) Grant No. 11871264,  Natural Science Foundation of Guangdong Province (2018A0303130123), and NSFC/Hong Kong RRC Joint Research Scheme (NFSC/RGC 11961160718).
The research of Z. Zhou is partially supported by the Hong Kong RGC grant (project no. 25300818).
\bibliographystyle{abbrv}

\end{document}